\documentclass[a4paper, 12pt]{article}
\usepackage{latexsym}
\usepackage{amssymb}
\usepackage{amsthm}
\usepackage{amsmath}

\newtheorem{theorem}{Theorem}[section]

\theoremstyle{theorem}
\newtheorem{definition}[theorem]{Definition}

\newtheorem{proposition}[theorem]{Proposition}
\newtheorem{corollary}[theorem]{Corollary}

\theoremstyle{remark}

\numberwithin{equation}{section}

\makeatletter
\def\Ddots{\mathinner{\mkern1mu\raise\p@
\vbox{\kern7\p@\hbox{.}}\mkern2mu
\raise4\p@\hbox{.}\mkern2mu\raise7\p@\hbox{.}\mkern1mu}}
\makeatother

\def\a{\alpha}

\def\la{\lambda}

\title{Factorial Schur functions via the six vertex model}
\author{Peter J. McNamara\\
\small Department of Mathematics\\[-0.8ex]
\small Massachusetts Institute of Technology, MA 02139, USA\\[-0.8ex]
\small \texttt{petermc@math.mit.edu}}

\date{October 31, 2009}

\begin{document}

\maketitle

\abstract{For a particular set of Boltzmann weights and a particular boundary condition for the six vertex model in statistical mechanics, we compute explicitly the partition function and show it to be equal to a factorial Schur function, giving a new proof of a theorem of Lascoux.}

\section{Introduction}

A state of the six vertex model, also called square ice, is defined as follows. Between each pair of neighbouring vertices of (a subset of) the square lattice in two spatial dimensions, we place an oriented edge, subject to the condition that at each vertex, exactly two of the four edges are oriented inbound. Around each vertex there are ${4 \choose 2}=6$ potential configurations, hence the name. We shall always consider states of the six vertex model that are bounded spatially with fixed boundary conditions. One can also realise this model in terms of a square array of ice molecules, where an oxygen atom is placed at each vertex, and a hydrogen atom is placed on each edge, with the covalent bonding determined by the hydrogen atom being bonded to the source oxygen atom of its orientation. It is the ice interperetation that we shall use in this paper. In diagrams, we represent oxygen atoms by circles and hydrogen atoms by smaller solid disks, with the existence of covalent bonds being exhibited via spacial proximity in diagrams.

To each vertex $v$ in a state $T$, we associate a Boltzmann weight $w_v(T)$ which depends only on the location of $v$ in the lattice and the local configuration of incoming and outgoing edges containing the vertex $v$. The partition function (from statistical mechanics) for this model is then defined to be
$$
Z=\sum_T\prod_v w_v(T),
$$
where the sum is taken over all possible states $T$ and the product is over all vertices $v$ in the model. It is this function which we intend to evaluate in this paper.

Let $n$ be a positive integer and $\la$ a partition with $n$ parts. So $\la=(\la_1,\la_2,\ldots,\la_n)$ where $\la_1\geq\cdots\geq\la_n\geq 0$ are all integers. We consider a rectangular array of square ice consisting of $n$ rows and $n+\la_1$ columns (the effect of adding more columns to the right will trivially have no effect on the partition function). On this model, we impose the following boundary conditions, which are the same as found in \cite{bbf}. The left and right hand boundaries will be filled with hydrogen atoms with the lower boundary completely devoid of hydrogen atoms. These conditions already imply that along the upper boundary exactly $n$ of the columns will not have hydrogen atoms, we insist that these be in the columns $\la_i+n+1-i$ for $i=1,\ldots,n$. We shall call this boundary condition the $\la$-boundary condition. An example of such a configuration with $\la=(5,4,1)$ is

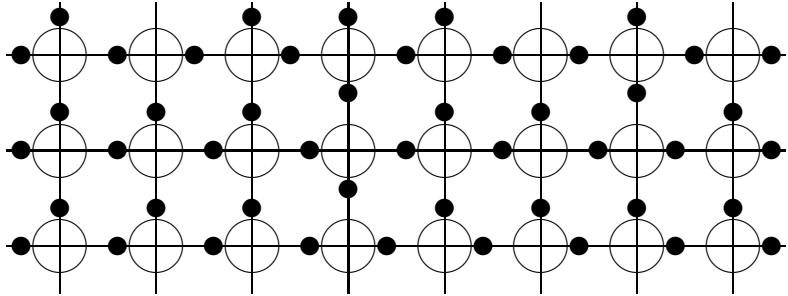
\[
\setlength{\unitlength}{1.8pt}
\begin{picture}(160,60)
\put(-1,30){\line(1,0){162}}
\put(-1,50){\line(1,0){162}}
\put(-1,10){\line(1,0){162}}
\put(10,0){\line(0,1){61}}
\put(30,0){\line(0,1){61}}
\put(50,0){\line(0,1){61}}
\put(70,0){\line(0,1){61}}
\put(90,0){\line(0,1){61}}
\put(110,0){\line(0,1){61}}
\put(130,0){\line(0,1){61}}
\put(150,0){\line(0,1){61}}
\put(10,10){\circle{10}}
\put(10,30){\circle{10}}
\put(10,50){\circle{10}}
\put(30,10){\circle{10}}
\put(30,30){\circle{10}}
\put(30,50){\circle{10}}
\put(50,10){\circle{10}}
\put(50,30){\circle{10}}
\put(50,50){\circle{10}}
\put(70,10){\circle{10}}
\put(70,30){\circle{10}}
\put(70,50){\circle{10}}
\put(90,10){\circle{10}}
\put(90,30){\circle{10}}
\put(90,50){\circle{10}}
\put(110,10){\circle{10}}
\put(110,30){\circle{10}}
\put(110,50){\circle{10}}
\put(130,10){\circle{10}}
\put(130,30){\circle{10}}
\put(130,50){\circle{10}}
\put(150,10){\circle{10}}
\put(150,30){\circle{10}}
\put(150,50){\circle{10}}
\put(2,10){\circle*{4}}\put(2,50){\circle*{4}}\put(2,30){\circle*{4}}
\put(158,10){\circle*{4}}
\put(158,50){\circle*{4}}
\put(158,30){\circle*{4}}
\put(10,58){\circle*{4}}\put(50,58){\circle*{4}}\put(70,58){\circle*{4}}\put(90,58){\circle*{4}}\put(130,58){\circle*{4}}
\put(70,42){\circle*{4}}
\put(70,22){\circle*{4}}
\put(130,42){\circle*{4}}
\put(10,38){\circle*{4}}
\put(30,38){\circle*{4}}
\put(50,38){\circle*{4}}
\put(150,38){\circle*{4}}
\put(90,38){\circle*{4}}
\put(110,38){\circle*{4}}\put(10,18){\circle*{4}}
\put(30,18){\circle*{4}}
\put(50,18){\circle*{4}}
\put(150,18){\circle*{4}}
\put(90,18){\circle*{4}}
\put(110,18){\circle*{4}}
\put(130,18){\circle*{4}}
\put(22,10){\circle*{4}}\put(22,50){\circle*{4}}\put(22,30){\circle*{4}}
\put(42,10){\circle*{4}}\put(38,50){\circle*{4}}\put(42,30){\circle*{4}}
\put(62,10){\circle*{4}}\put(58,50){\circle*{4}}\put(62,30){\circle*{4}}
\put(78,10){\circle*{4}}\put(82,50){\circle*{4}}\put(82,30){\circle*{4}}
\put(98,10){\circle*{4}}\put(102,50){\circle*{4}}\put(102,30){\circle*{4}}
\put(118,10){\circle*{4}}\put(118,50){\circle*{4}}\put(122,30){\circle*{4}}
\put(138,10){\circle*{4}}\put(142,50){\circle*{4}}\put(138,30){\circle*{4}}
\end{picture}
\setlength{\unitlength}{1pt}
\]

Let $x_1,x_2,\ldots,x_n$ and $a_1,a_2,\ldots,a_{n+\la_1}$ be two alphabets of variables.
For computing the partition function, we use the following set of Boltzmann weights. For the vertical molecule in row $i$ and column $j$, we attach the weight $x_i/a_j$. For the molecule with hydrogen atoms to the north and west, we attach the weight $x_i/a_j-1$. For all other orientations of molecules, we attach the weight 1. Let $Z_\la(x|a)$ denote the corresponding partition function. Our main theorem is the following.

\begin{theorem}\label{main}\cite[Theorem 1]{lascoux}
Up to a monomial, the partition function $Z_\la$ is equal to a factorial Schur function. Explicitly, we have
\[
Z_\la(x|a)=\frac{x^\delta}{a^{(\la+\rho)'}}s_\la(x|a).
\]
\end{theorem}

The notation of the above, used through out this paper, is that whenever $y=(y_1,y_2,\ldots)$ and $\mu=(\mu_1,\mu_2,\ldots)$ is a multi-index, then $y^\mu=y_1^{\mu_1}y_2^{\mu_2}\cdots$. We define the partition $\rho$ to be $(n-1,n-2,\ldots,0)$ and the multi-index $\delta$ as $(0,1,\ldots,n-1)$. For any partition $\la$, $\la'$ denotes its conjugate (the reflection of $\la$, identified with its Young diagram, in the main diagonal). A cell $(i,j)\in\la$ refers to the box in the intersection of the $i$-th row and $j$-th column of the Young diagram of $\la$.

Our proof is fundamentally different to that of \cite{lascoux}. We begin by showing the symmetry of $x^{-\delta}Z_\la(x|a)$ and then use the characteristic vanishing properties of factorial Schur functions to identify the partition function as a factorial Schur function.

The author would like to thank Ben Brubaker, Alain Lascoux and Steven Sam for helpful conversations.

\section{Factorial Schur functions}
We present an overview of the definition and basic properties of factorial Schur functions for which \cite{molev} can be taken as a reference.

\begin{definition}
The factorial Schur function $s_\la(x|a)$ is defined to be
\[
s_\la(x|a)=\sum_T \prod_{\a\in \la} (x_{T(\a)}-a_{T(\a)+c(\a)})
\]
where $c(\a)$ is the content of the cell $\a$ in the Young diagram of $\la$, defined by $c(i,j)=j-i$. The summation is taken over all semistandard Young tableaux $T$ of shape $\la$.
\end{definition}

We now state the primary properties of these factorial Schur functions without proof.

\begin{proposition}
The factorial Schur function $s_\la(x|a)$ is symmetric in the variables $x_1,\ldots,x_n$.
\end{proposition}

\begin{proposition}
The collection of factorial Schur functions $s_\la(x|a)$ as $\la$ runs over all partitions with $n$ parts, forms a basis for the ring of symmetric polynomials in $x_1,\ldots,s_n$. Restricting to symmetric polynomials of degree at most $m$, the collection of factorial Schur functions $s_\la(x|a)$ as above with the additional restriction that $|\la|=\sum_i\la_i\leq m$ forms a basis for this vector space.
\end{proposition}

Given a partition $\mu=(\mu_1,\ldots,\mu_n)$, define the sequence $a_\mu=((a_\mu)_1,\ldots,(a_\mu)_n)$ by $(a_\mu)_i=a_{n+1-i+\la_i}$.

\begin{proposition}[Vanishing Theorem]
Suppose that $\la$ and $\mu$ are partitions of at most $n$ parts. Then $$s_\la(a_\mu|a)=0 \text{ if }\la\not\subset\mu,$$ $$s_\la(a_\la|a)=\prod_{(i,j)\in\la} (a_{n+1-i+\la_i}-a_{n-\la_j'+j}).$$
\end{proposition}


\section{Strict Gelfand-Tsetlin Patterns and Staircases}
To understand the possible states of square ice with the $\la$-boundary condition, and to connect our combinatorial setup with the staircase language of Lascoux \cite{lascoux}, it will be useful to consider a bijection between such states and two other combinatorial objects, namely strict Gelfand-Tsetlin patterns with top row $\la+\rho$ and staircases with rightmost column missing $\la+\rho$. We now define these other families of objects.

A Gelfand-Tsetlin pattern is a triangular array of natural numbers
\[
\mathfrak{T}=\left\{
\begin{array}{cccccccccc}
t_{1,1} & & t_{1,2} & & t_{1,3} & & \ldots & t_{1,n-1} & & t_{1,n}  \\
 & t_{2,1} & & t_{1,2} & & t_{2,3} & \ldots &  & t_{2,n-1} & \\
 & & \ddots & & & & & \Ddots  & & \\
 & & & & & t_{n,1} & & & &
\end{array}
 \right\}
\]
subject to the inequalities $t_{i,j}\leq t_{i+1,j}\leq t_{i,j+1}$ for all $i$ and $j$. This pattern is further called strict if the additional inequalities $t_{i,j}<t_{i,j+1}$ also hold.

To say that the top row of $\mathfrak{T}$ is equal to $\la+\rho$ is to say that $t_{1,n+1-i}=n+1-i+\la_i$ for all $i$.

A staircase is a semistandard Young tableau of shape $(\la_1+n,\la_1+n-1,\ldots,\la_1)'$ filled with numbers from $\{1,2,\ldots,\la_1+n\}$ with the additional condition that the diagonals are weakly decreasing in the south-east direction when written in the French notation.

To say that a staircase has rightmost column missing $\la+\rho$ means that the integers in the rightmost column comprise all integers from $1$ to $\la_1+n$ inclusive with the exception of those integers of the form $n+1-i+\la_i$ for some $i$.

An example of a staircase for $\la=(5,4,1)$ is

\[
\setlength{\unitlength}{1.46em}
\begin{picture}(4,8)
\multiput(0,0)(0,1){6}{\line(1,0){4}}
\multiput(1,8)(1,-1){4}{\line(0,-1){3}}
\multiput(0,0)(1,0){5}{\line(0,1){5}} \put(0,8){\line(1,0){1}}
\put(0,7){\line(1,0){2}}
\put(0,6){\line(1,0){3}}
\put(0,8){\line(0,-1){3}}
\multiput(0.35,0.28)(1,0){4}{1}
\multiput(0.35,1.28)(1,0){3}{2}
\multiput(0.35,2.28)(1,0){3}{3}
\multiput(1.35,3.28)(1,0){3}{5}
\multiput(0.35,7.28)(1,-1){3}{8}
\put(3.35,4.28){7}
\multiput(0.35,6.28)(1,-1){2}{7}
\multiput(0.35,5.28)(1,-1){2}{6}
\multiput(0.35,4.28)(1,-1){2}{5}
\put(0.35,3.28){4} \put(3.35,2.28){4}
\put(2.35,4.28){6} \put(3.35,1.28){3}
\end{picture}
\setlength{\unitlength}{1pt}
\]
We have chosen to present this particular example because under the bijection we are about to construct, it gets mapped to the sample array of square ice given in the introduction.

\begin{proposition}
Let $\la$ be a partition. There are natural bijections between the following three sets of combinatorial objects
\begin{enumerate}
\item States of the six vertex model with the $\la$ boundary condition,
\item Strict Gelfand-Tsetlin patterns with top row $\la+\rho$, and
\item Staircases whose rightmost column is missing $\la+\rho$.
\end{enumerate}
\end{proposition}

\begin{proof}
For each state of the six vertex model, and each $i=2,\ldots,n$, there are $n+1-i$ occurrences of an oxygen atom between rows $i-1$ and $i$ covalently bonded to the hydrogen atom in row $i-1$ above it. If we let $t_{i,j}$ for $j=1,\ldots,n+1-i$ be the column numbers of these oxygen atoms, then this corresponding collection of integers forms a strict Gelfand-Tsetlin pattern with top row $\la+\rho$.

Given a strict Gelfand-Tsetlin pattern $\mathfrak{T}$ with top row $\la+\rho$, we construct a staircase whose rightmost column is missing $\la+\rho$ in the following manner: We fill column $j+1$ with integers $u_1,\ldots,u_{n-j+\la_1}$ such that
\[
\{1,2,\ldots,n+\la_1 \} = \{u_1,\ldots ,u_{n-j+\la_1} \} \sqcup \{ t_{n+1-j,1},\ldots t_{n+1-j,j-1} \}.
\]
It is easily checked that these maps give the desired bijections.
\end{proof}

With these bijections in hand, it is a straightforward manner to check that our formulation of Theorem \ref{main} is equivalent to \cite[Theorem 1]{lascoux}, noting that a double Schubert polynomial for a Grassmannian permutation is the same as a factorial Schur function.

\section{The star triangle relation}


We consider a set of Boltzmann weights on a diagonally oriented vertex, corresponding to string numbers $i$ and $j$ where $i$ runs from lower left to top right.

For each of the six admissible orientations of the ice molecule, we attach the following weights: 
\[
\setlength{\unitlength}{1.4pt}
\begin{picture}(12,6)
\put(6,0){\circle{10}}
\put(0.3,-5.7){\circle*{4}}
\put(11.7,5.7){\circle*{4}}
\end{picture}
\setlength{\unitlength}{1pt}
 \  :\frac{1}{x_i}\,,\quad
\setlength{\unitlength}{1.4pt}
\begin{picture}(12,6)
\put(6,0){\circle{10}}
\put(0.3,5.7){\circle*{4}}
\put(11.7,-5.7){\circle*{4}}
\end{picture}
\setlength{\unitlength}{1pt}
 \ :\frac{1}{x_j}\,,\quad\setlength{\unitlength}{1.4pt}
\begin{picture}(12,6)
\put(6,0){\circle{10}}
\put(0.3,5.7){\circle*{4}}
\put(11.7,5.7){\circle*{4}}
\end{picture}
\setlength{\unitlength}{1pt}
 \ :\left(\frac{1}{x_i}-\frac{1}{x_j}\right),\quad\setlength{\unitlength}{1.4pt}
\begin{picture}(12,6)
\put(6,0){\circle{10}}
\put(0.3,-5.7){\circle*{4}}
\put(11.7,-5.7){\circle*{4}}
\end{picture}
\setlength{\unitlength}{1pt}
 \ :0,\!\quad\setlength{\unitlength}{1.4pt}
\begin{picture}(12,6)
\put(6,0){\circle{10}}
\put(11.7,-5.7){\circle*{4}}
\put(11.7,5.7){\circle*{4}}
\end{picture}
\setlength{\unitlength}{1pt}
\ :\frac{1}{x_i}\,,\quad\setlength{\unitlength}{1.4pt}
\begin{picture}(12,6)
\put(6,0){\circle{10}}
\put(0.3,-5.7){\circle*{4}}
\put(0.3,5.7){\circle*{4}}
\end{picture}
\setlength{\unitlength}{1pt}
:\frac{1}{x_j}.
\]

Now consider the following diagrams

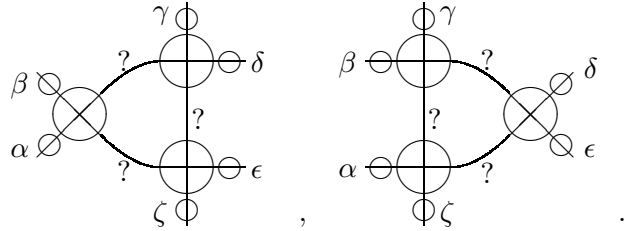
\[
\setlength{\unitlength}{2pt}
\begin{picture}(40,40)
\put(2,12){\line(1,1){12}} \put(2,28){\line(1,-1){12}}
\put(25,10){\line(1,0){16}} \put(25,30){\line(1,0){16}}
\put(30,-1){\line(0,1){42}}
\put(10,20){\circle{10}}
\put(30,30){\circle{10}}
\put(30,10){\circle{10}}
\put(30,38){\circle{4}}
\put(38,10){\circle{4}}
\put(4.3,25.7){\circle{4}}
\put(30,2){\circle{4}}
\put(4.3,14.3){\circle{4}}
\put(38,30){\circle{4}}
\qbezier(14,24)(20,30)(25,30)
\qbezier(14,16)(20,10)(25,10)
\put(-3,12){$\alpha$}
\put(-3,24){$\beta$}
\put(23.6,38){$\gamma$}
\put(42,28){$\delta$}
\put(42,8){$\epsilon$}
\put(23.6,0){$\zeta$}
\put(17,28.2){?}\put(17,7.52){?}\put(30.9,17.2){?}
\end{picture}\qquad,\qquad
\begin{picture}(40,40)
\put(38,12){\line(-1,1){12}} \put(38,28){\line(-1,-1){12}}
\put(-1,10){\line(1,0){16}} \put(-1,30){\line(1,0){16}}
\put(10,-1){\line(0,1){42}}
\put(30,20){\circle{10}}
\put(10,30){\circle{10}}
\put(10,10){\circle{10}}
\put(10,38){\circle{4}}
\put(2,10){\circle{4}}
\put(35.7,25.7){\circle{4}}
\put(10,2){\circle{4}}
\put(35.7,14.3){\circle{4}}
\put(02,30){\circle{4}}
\qbezier(15,10)(20,10)(26,16)
\qbezier(15,30)(20,30)(26,24)
\put(-6,8.1){$\alpha$}
\put(-6,28){$\beta$}
\put(13,38){$\gamma$}
\put(40,26.8){$\delta$}
\put(40,12){$\epsilon$}
\put(13,0){$\zeta$}
\put(20.6,28.4){?}\put(20.6,7){?}\put(10.8,17.3){?}
\end{picture}\qquad\!\!\!\!\!.
\setlength{\unitlength}{1pt}
\]

We consider boundary conditions $\a$, $\beta$, $\gamma$, $\delta$, $\epsilon$ and $\zeta$, exactly three of which correspond to the placement of a hydrogen atom in the corresponding place in the above diagrams. With these boundary conditions, we consider all possible ways of completing this to an arrangement of square ice by the placement of exactly one hydrogen atom along each string marked with a question mark. Each such state has a corresponding Boltzmann weight, where for the diagonal molecules we attach the weights presented immediately above while for the rectilinear molecules, we attach the same Boltzmann weight as in the introduction, where the row number is determined by the string it is sitting on.

\begin{theorem}[Star Triangle Identity]
The star triangle identity holds in the following sense. We fix boundary conditions and sum the Boltzmann weights of all possible states of the left hand diagram. Then this equals the sum of all the Boltzmann weights of all possible states of the right hand diagram with the same boundary conditions.
\end{theorem}

\begin{proof}
There are a total of ${6\choose 3}=20$ possible boundary conditions for which this theorem needs to be checked. We will illustrate this with a particular example.

Consider the boundary condition consisting of a hydrogen atom in the $\beta$, $\gamma$ and $\zeta$ positions. Then our purported identity becomes

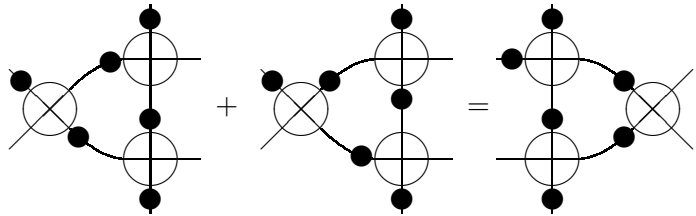
\[
\setlength{\unitlength}{1.88pt}
\begin{picture}(140,40)
\put(2,12){\line(1,1){12}} \put(2,28){\line(1,-1){12}}
\put(25,10){\line(1,0){15}} \put(25,30){\line(1,0){15}}
\put(30,-1){\line(0,1){42}}
\put(10,20){\circle{10}}
\put(30,30){\circle{10}}
\put(30,10){\circle{10}}
\put(30,38){\circle*{4}}
\put(30,18){\circle*{4}}
\put(4.3,25.7){\circle*{4}}
\put(30,2){\circle*{4}}
\put(15.7,14.3){\circle*{4}}
\put(22,29.5){\circle*{4}}
\put(43,19){$+$}
\qbezier(14,24)(20,30)(25,30)
\qbezier(14,16)(20,10)(25,10)
\put(52,12){\line(1,1){12}} \put(52,28){\line(1,-1){12}}
\put(75,10){\line(1,0){15}} \put(75,30){\line(1,0){15}}
\put(80,-1){\line(0,1){42}}
\put(60,20){\circle{10}}
\put(80,30){\circle{10}}
\put(80,10){\circle{10}}
\put(80,38){\circle*{4}}
\put(80,22){\circle*{4}}
\put(54.3,25.7){\circle*{4}}
\put(80,2){\circle*{4}}
\put(65.7,25.7){\circle*{4}}
\put(72,10.5){\circle*{4}}
\qbezier(64,24)(70,30)(75,30)
\qbezier(64,16)(70,10)(75,10)
\put(93,19){=}
\put(138,12){\line(-1,1){12}} \put(138,28){\line(-1,-1){12}}
\put(99,10){\line(1,0){16}} \put(99,30){\line(1,0){16}}
\put(110,-1){\line(0,1){42}}
\put(130,20){\circle{10}}
\put(110,30){\circle{10}}
\put(110,10){\circle{10}}
\put(110,38){\circle*{4}}
\put(110,18){\circle*{4}}
\put(124.3,25.7){\circle*{4}}
\put(110,2){\circle*{4}}
\put(124.3,14.3){\circle*{4}}
\put(102,30){\circle*{4}}
\qbezier(115,10)(120,10)(126,16)
\qbezier(115,30)(120,30)(126,24)
\end{picture}
\setlength{\unitlength}{1pt}
\]

$$
(\frac{1}{x_i}-\frac{1}{x_j})\frac{x_i}{a}+\frac{1}{x_j}(\frac{x_i}{a}-1)\frac{x_j}{a}=(\frac{x_j}{a}-1)\frac{x_i}{a}\frac{1}{x_j}.
$$
The other 19 cases (the majority of which are simpler than this) are left to the reader.
\end{proof}

\begin{corollary}
We have the identity
$$
\frac{1}{x_{i+1}}Z_\la(\ldots,x_i,x_{i+1},\ldots|a)=\frac{1}{x_i}Z_\la(\ldots,x_{i+1},x_i,\ldots|a).
$$
\end{corollary}
\begin{proof}
The left hand side is the partition function of our square ice model, with a crossing between horizonal strings $i$ and $i+1$ added on the left. All boundary conditions are the same as for the calculation of $Z_\la$. The right hand side is the partition function for the same ice model, but this time the crossing between strings $i$ and $i+1$ is added on the right. Using the star triangle relation proved above, we can incrementally push the crossing of horizonal strings past each vertical string while preserving the partition function, hence proving the Corollary.
\end{proof}

\begin{corollary}
The function $x^{-\delta}Z_\la(x|a)$ is symmetric in $x_1,\ldots,x_n$.
\end{corollary}

\begin{corollary}
The function $x^{-\delta}Z_\la(x|a)$ is a polynomial in $x_1,\ldots,x_n$.
\end{corollary}
\begin{proof}
A priori, we know this to be a Laurent polynomial in $x_1,\ldots,x_n$ that is a polynomial in $x_1$. Since it is symmetric, it must thus be a polynomial in $x_1,\ldots,x_n$.
\end{proof}

We remark that the symmetry result proved in this section does not use the upper or lower boundary conditions imposed on our states of square ice, and thus gives a direct proof of the symmetry property of the function $F(u,v;{\bf z})$ from \cite[Theorem 2]{lascoux}.

\section{Proof of the Main Theorem}
Our method of proof is to show that the function $Z_\la(x|a)$ possesses the same characteristic vanishing properties of the factorial Schur functions.

\begin{theorem}
For two partitions $\la$ and $\mu$ of at most $n$ parts, we have
$$Z_\la(a_\mu|a)=0 \text{ unless } \la\subset\mu ,$$
$$Z_\la(a_\la|a)=\prod_{(i,j)\in\la} \left(\frac{a_{n+1-i+\la_i}}{a_{n-\la'_j+j}}-1\right).$$
\end{theorem}
\begin{proof}
Fix a state, and assume that this state gives a non-zero contribution to the partition function $Z_\la(a_\mu|a)$. Under the bijection between states of square ice and strict Gelfand-Tsetlin patterns, let $k_i$ be the rightmost entry in the $i$-th row of the corresponding Gelfand-Tsetlin pattern. We shall prove by descending induction on $i$ the inequality
$$
n+1-i+\mu_i\geq k_i.
$$
For any $j$ such that $k_{i+1}<j<k_i$, there is a factor $(x_i/a_j-1)$ in the Boltzmann weight of this state. We have the inequality $n+1-i+\mu_i>n+1-(i+1)+\mu_{i+1}\geq k_{i+1}$ by our inductive hypothesis. Since $(a_\mu)_i=a_{n+1-i+\la_i}$, in order for this state to give a non-zero contribution to $Z_\la(a_\mu|a)$, we must have that $n+1-i+\mu_i\geq k_i$, as required.

Note that for all $i$, we have $k_i\geq n+1-i+\la_i$. Hence $\mu_i\geq\la_i$ for all $i$, showing that $\mu\supset\la$ as required, proving the first part of the theorem.

To compute $Z_\la(a_\la|a)$, notice that the above argument shows that there is only one state which gives a non-zero contribution to the sum. Under the bijection with Gelfand-Tsetlin patterns, this is the state with $t_{i,j}=t_{1,j}$ for all $i,j$. The formula for $Z_\la(a_\la|a)$ is now immediate.
\end{proof}

Now we have collated enough data to be able to prove Theorem \ref{main}.

\begin{proof} We can see, for example using the bijection with Gelfand-Tsetlin patterns, that $Z_\la$ is a polynomial of degree at most $|\la+\rho|$ in $x$. Thus $Z_\la/x^{\delta}$ is a symmetric polynomial of degree at most $|\la|$ in $x$, so we may write it in the form.
$$
x^{-\delta}Z_\la=\sum_{\mu:|\mu|\leq |\la|} c_\mu(a)s_\mu(x|a)$$
for some coefficients $c_\mu$ which are rational functions of the $a$ variables only.

Inducting on $\mu$, we substitute $x=a_\mu$ into the above equation to conclude that $c_\mu(a)=0$ for all $\mu\neq\la$ with $|\mu|\leq |\la|$. By substituting $x=a_\la$, we are also able to calculate $c_\la(a)=(a^{(\la+\delta)'})^{-1}$, completing the proof.\end{proof}

\end{document}